\documentclass[twoside, 11pt]{article}

\usepackage{graphicx, subfigure}
\usepackage[top=1in,bottom=1in,left=1in,right=1in,footnotesep=0.5in]{geometry}
\usepackage[sort&compress, numbers]{natbib} 
\usepackage{amsfonts, amsmath, amssymb, amsthm}
\usepackage{MnSymbol}
\usepackage{mathtools}
\mathtoolsset{showonlyrefs}
\usepackage{enumitem}
\usepackage{comment}

\usepackage{sectsty}
\sectionfont{\fontsize{12}{15}\selectfont}

\usepackage{color}
\definecolor{darkred}{RGB}{100,0,0}
\definecolor{darkgreen}{RGB}{0,100,0}
\definecolor{darkblue}{RGB}{0,0,150}

\usepackage{hyperref}
\hypersetup{colorlinks=true, linkcolor=darkred, citecolor=darkgreen, urlcolor=darkblue}
\usepackage{url}

\def\level{\cL}
\def\up{\cU}

\def\I{{\rm I}}
\def\d{{\rm d}}

\def\ball{B}

\newtheorem{thm}{Theorem}[section]

\newtheorem{prp}[thm]{Proposition}

\newtheorem{lem}[thm]{Lemma}

\theoremstyle{remark}

\newtheorem{rem}[thm]{Remark}

\def\beq{\begin{equation}} 
\def\eeq{\end{equation}}
\def\beqn{\begin{eqnarray*}}
\def\eeqn{\end{eqnarray*}}
\def\Bitem{\begin{itemize}\setlength{\itemsep}{.2in}}
\def\bitem{\begin{itemize}\setlength{\itemsep}{.05in}}
\def\eitem{\end{itemize}}
\def\Benum{\begin{enumerate}\setlength{\itemsep}{.2in}}
\def\benum{\begin{enumerate}\setlength{\itemsep}{.05in}}
\def\eenum{\end{enumerate}}
\def\bmult{\begin{multline*}}
\def\emult{\end{multline*}}
\def\bcenter{\begin{center}}
\def\ecenter{\end{center}}
\def\bframe{\begin{frame}}
\def\eframe{\end{frame}}

\newcommand{\thmref}[1]{Theorem~\ref{thm:#1}}
\newcommand{\prpref}[1]{Proposition~\ref{prp:#1}}

\newcommand{\lemref}[1]{Lemma~\ref{lem:#1}}
\newcommand{\secref}[1]{Section~\ref{sec:#1}}
\newcommand{\figref}[1]{Figure~\ref{fig:#1}}

\DeclareMathOperator{\dist}{dist}

\def\cA{\mathcal{A}}
\def\cB{\mathcal{B}}
\def\cC{\mathcal{C}}

\def\cG{\mathcal{G}}
\def\cH{\mathcal{H}}

\def\cL{\mathcal{L}}

\def\cT{\mathcal{T}}
\def\cU{\mathcal{U}}
\def\cV{\mathcal{V}}
\def\cW{\mathcal{W}}

\def\bbR{\mathbb{R}}

\def\1{\mathbbm{1}}

\definecolor{purple}{rgb}{0.4,.1,.9}

\begin{document}
\thispagestyle{empty}

\title{Moving Up the Cluster Tree with the Gradient Flow}
\author{
Ery Arias-Castro\footnote{University of California, San Diego, California, USA (\url{https://math.ucsd.edu/\~eariasca/})} 
\and 
Wanli Qiao\footnote{George Mason University, Fairfax, Virginia, USA (\url{https://mason.gmu.edu/\~wqiao/})}
}
\date{}
\maketitle

\begin{abstract}
The paper establishes a strong correspondence between two important clustering approaches that emerged in the 1970's: clustering by level sets or cluster tree as proposed by Hartigan and clustering by gradient lines or gradient flow as proposed by Fukunaga and Hostetler. We do so by showing that we can move up the cluster tree by following the gradient ascent flow.

\medskip\noindent
{\em Keywords and phrases:}
clustering; level sets; cluster tree; gradient lines; gradient flow; mean-shift algorithm; dynamical systems
\end{abstract}

\section{\uppercase{Introduction}} 
\label{sec:intro}

Up until the 1970's there were two main ways of clustering points in space. 
One of them, perhaps pioneered by \citet{pearson1894contributions}, was to fit a (usually Gaussian) mixture to the data, and that being done, classify each data point --- as well as any other point available at a later date --- according to the most likely component in the mixture. 
The other one was based on a direct partitioning of the space, most notably by minimization of the average minimum squared distance to a center: the $K$-means problem, whose computational difficulty led to a number of famous algorithms \citep{macqueen1967some, forgy1965, hartigan1979k, lloyd1982least, max1960quantizing} and likely played a role in motivating the development of hierarchical clustering \citep{ward1963hierarchical, fisher1958grouping, gower1969minimum, sneath1957application}.

In the 1970's, two decidedly nonparametric approaches to clustering were proposed, both based on the topography given by the population density. Of course, in practice, the density is estimated, often by some form of kernel density estimation.

\paragraph{Clustering via level sets}
One of these approaches is that of \citet{hartigan1975clustering}, who proposed to look at the connected components of the upper-level sets of the population density.
Thinking of clusters as ``regions of high density separated from other such regions by regions of low density", at a given level, each connected component represents a cluster, while the remaining region in space is sometimes considered as noise. 
The basic idea was definitely in the zeitgeist. For example, a similar approach was suggested around the same time by \citet{koontz1972nonparametric}.

The choice of level is not at all obvious, and in fact Hartigan recommended looking at the entire tree structure --- which he called the ``density-contour tree'' and is now better known as the {\em cluster tree} --- that arises by the nesting property of the upper-level sets considered as a whole. 
Note, however, that the cluster tree does not provide a complete partitioning of the space. 

\citet{hartigan1981consistency, hartigan1977distribution, hartigan1977clusters}, and later \cite{penrose1995single}, showed that the cluster tree can be estimated by single linkage, achieving a weak notion of consistency called fractional consistency. Since then, the estimation of cluster trees using different algorithms or notions of consistency has been studied in \citep{stuetzle2003estimating, stuetzle2010generalized, rinaldo2012stability, chaudhuri2014consistent, wang2019dbscan, eldridge2015beyond}. At a fixed level in a cluster tree, clustering is naturally related to level set estimation, which has in itself received a lot of attention in the literature, e.g., \citep{polonik1995measuring, tsybakov1997nonparametric,rigollet2009optimal,chen2017density,mason2009asymptotic,walther1997granulometric,rinaldo2010generalized,singh2009adaptive,qiao2021nonparametric,mammen2013confidence,qiao2019nonparametric}. To address the problem of choosing the level, \citep{sriperumbudur2012consistency, steinwart2015fully, steinwart2011adaptive} considered the lowest split level in cluster tree, which can be used to recover the full cluster tree when applied recursively.

\paragraph{Clustering via gradient lines}
The other approach is that of \citet{fukunaga1975}, who proposed to use the gradient lines of the population density. Simply put, assuming the density has the proper regularity (which in particular requires that it is differentiable everywhere), a point is `moved' upward along the curve of steepest ascent in the topography given by the density, and the points ending at the same critical point form a cluster. 
This gradient flow definition of clustering is particularly relevant when the density is a Morse function \citep{chacon2015population}, as in that case each local maximum has its own basin of attraction and the union of these cover the entire space except for a set of Lebesgue measure zero.

The general idea of clustering by gradient ascent has been proposed or rediscovered a few times \citep{cheng2004estimating, li2007nonparametric, roberts1997parametric}.
For a fairly recent review of this literature, see \citep{carreira2015clustering}.
And the substantial amount of research work on density modes over the last several decades, which includes \citep{hartigan1985dip, silverman1981using, minnotte1997nonparametric, dumbgen2008multiscale, burman2009multivariate, genovese2016non},
is partly motivated by this perspective on clustering.

\paragraph{Contribution}
These two approaches --- by level sets and by gradient lines --- seem intuitively related, and in fact they are discussed together in a few recent review papers \citep{menardi2016review, chacon2020modal} under the umbrella name of {\em modal clustering}.
We argue here that the gradient flow is a natural way to partition the support of the density in concordance with the cluster tree. In doing so we provide a unified perspective on modal clustering, essentially equating the use of level sets and the use of gradient lines for the purpose of clustering points in space.


\paragraph{Setting}
Both approaches to clustering that we discuss here rely on features of an underlying density which throughout the paper will be denoted by $f$. Although $f$ is typically unknown, as the sample increases in size it becomes possible to estimate it consistently under standard assumptions, and the topographical features of $f$ that determine a clustering become estimable as well. In the spirit of \citep{chacon2015population}, for example, we focus on $f$ itself, which allows us to bypass technical finite sample derivations for the benefit of providing a more concise and clear picture.

\section{\uppercase{Clustering via level sets: the cluster tree}}
Given a density $f$ with respect to the Lebesgue measure on $\bbR^d$, for a positive real number $t > 0$, the $t$-upper level set of $f$ is given by
\begin{equation}
\up_t := \{x : f(x) \ge t\}.
\end{equation}
The level set of $f$ at level $t$ is defined by
\begin{equation}\label{Lt}
\level_t := \{x : f(x) = t\}.
\end{equation}
Throughout, whether specified or not, we will only consider levels that are in $(0, \max f)$. 

Hartigan, in his classic book on clustering, suggests that ``clusters may be thought of as regions of high density separated from other such regions by regions of low density" \cite[Sec 11.13]{hartigan1975clustering}. This naturally leads him to define clusters as the connected components of a certain upper level set of the underlying density: if the level is $t$, then the clusters are the connected components of $\up_t$ as defined above. 
See Figures~\ref{fig:level_sets_1d} and~\ref{fig:level_sets_2d} for illustrations in dimension~1 and~2, respectively.

\begin{figure}[!htpb]
\centering
\includegraphics[scale=0.6]{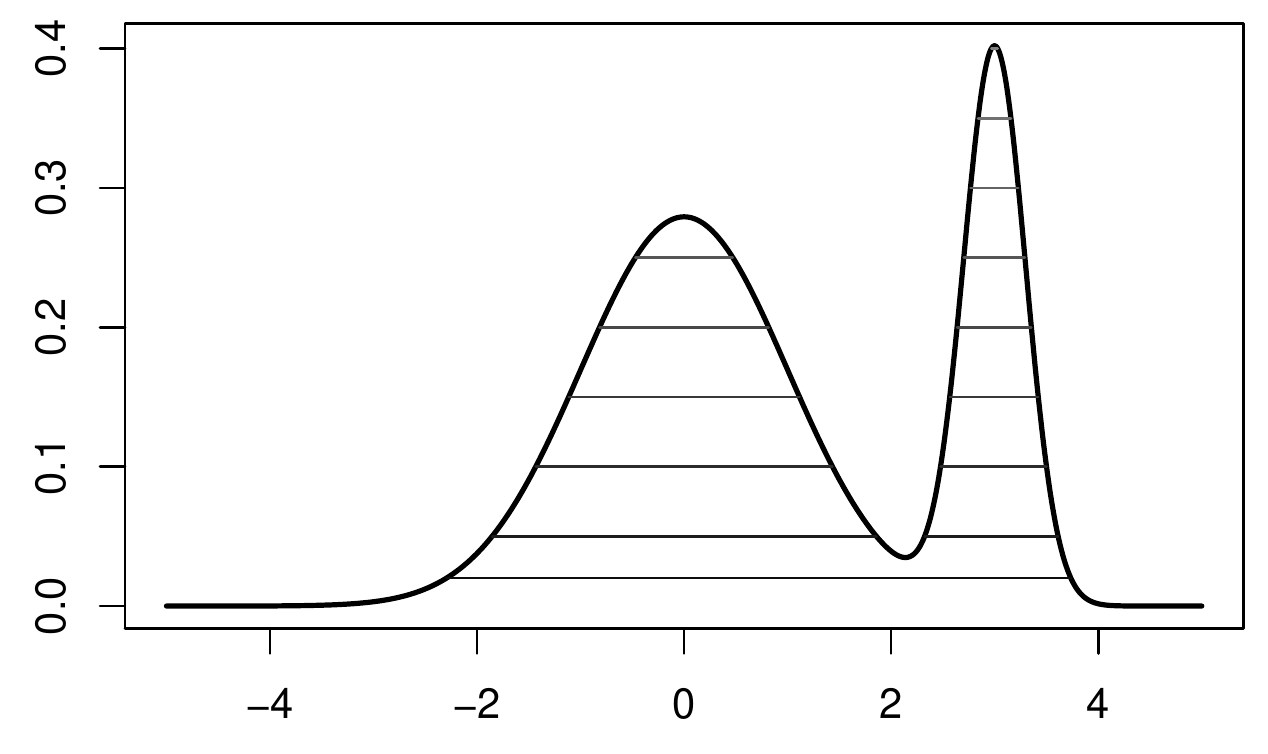}
\caption{A sample of upper level sets of a density in dimension $d = 1$ with two modes (which happens to be the mixture of two normal distributions). At any level $0 < t \le t_0$, where $t_0 \approx 0.0348$ is the value of the density at the local minimum, $\up_t$ is connected, and thus corresponds to the cluster at the level. At any level $t_0 < t \le t_1$, where $t_1 \approx 0.2792$ is the value of the density at its local maximum near $x=0$, $\up_t$ has exactly two connected components, and these are the clusters at that level. (At $t = t_1$, one of the clusters is a singleton.) Finally, at $t_1 < t \le t_2$, where $t_2 \approx 0.4021$ is the value of the density at its global maximum (near $x=3$), $\up_t$ is again connected, and is thus the cluster at that level. 
(At $t = t_2$, the cluster is a singleton.)}
\label{fig:level_sets_1d}
\end{figure}

\begin{figure}[!ht]
\centering
\includegraphics[scale=0.6]{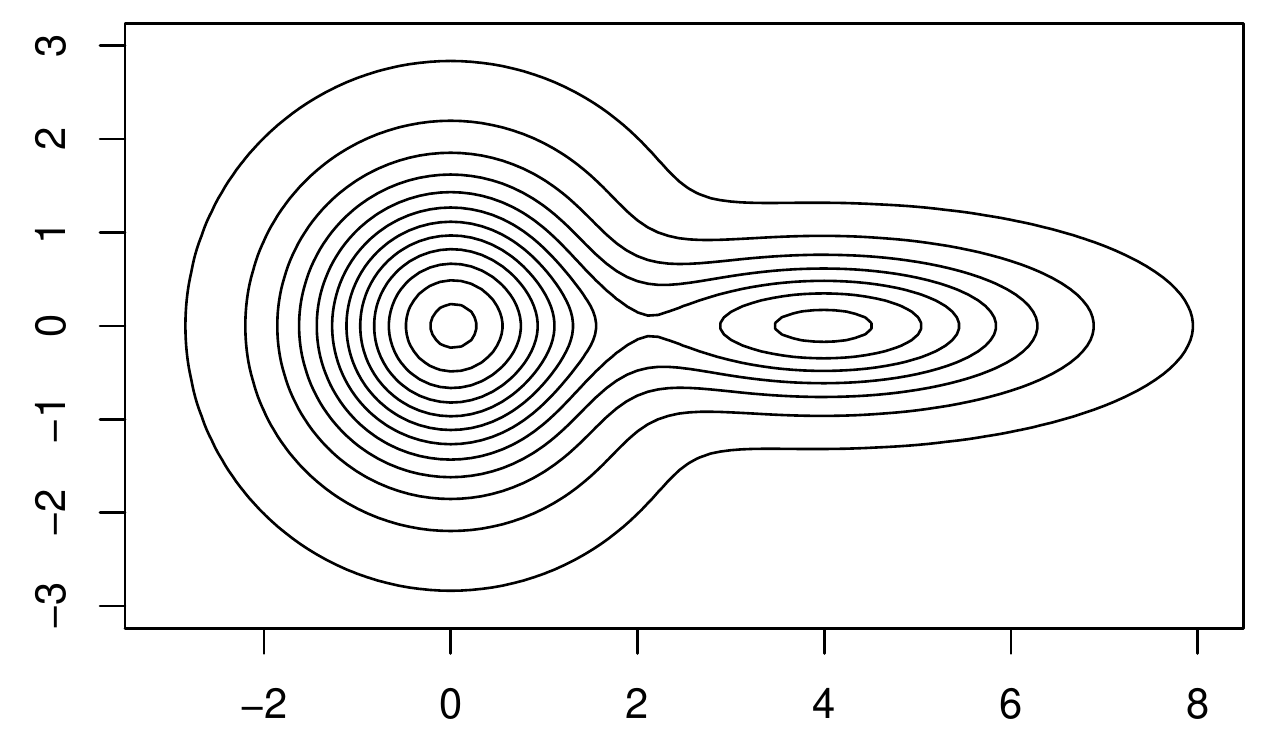}
\caption{A sample of upper level sets of a density in dimension $d = 2$ with two modes (which happens to be the mixture of two normal distributions, one with a non-scalar covariance matrix). The situation is similar to that of \figref{level_sets_1d}, where the number of connected components of the upper level set at $t$ is $=1$ one when $0 < t \le t_0$, where $t_0$ is the value of the density at the saddle point; $=2$ when $t_0 < t \le t_1$, where $t_1$ is the value of the density at the local (but not global) maximum; and $=1$ again when $t_1 < t \le t_2$, where $t_2$ is the maximum value of the density.} 
\label{fig:level_sets_2d}
\end{figure}

The choice of level $t$ is rather unclear in this definition, but can be determined by the number of clusters, which in turn is often set by the data analyst. Indeed, the situation is very much like that in hierarchical clustering: there is a tree structure. This structure comes from the nesting property of upper level sets where $\up_s \subset \up_t$ whenever $s \ge t$, which also implies that each cluster at level $s$ is included in a cluster at level $t$. The set of all cluster (each one being the connected component of an upper level set) equipped with this tree structure or partial ordering is what is called the {\em cluster tree}\footnote{~Here we consider the continuous cluster tree that includes all levels. In some other works, the levels are restricted to those where a topological change occurs.} --- and what Hartigan calls the ``density-contour tree''.
Note that the root represents the entire population while the leaves are the modes (i.e., local maxima).

The clusters at a particular level do not constitute a partition of the population. Indeed, regardless of $t > 0$, the clusters at level $t$, meaning the connected components of $\up_t$, form a partition of $\up_t$ itself, obviously, but not a partition of $\bbR^d$ since $\up_t \subsetneq \bbR^d$. And the cluster tree is only an organization of all the clusters at all levels, and thus also fails to provide a partition of the population.
According to a recent review paper by \citet{menardi2016review}, the region outside the upper level set of interest is dealt with via (supervised) classification. Suppose the level is chosen in some way, perhaps according to the desired number of clusters, and denoted $t$. The connected components of $\up_t$ are then computed. Then each point in $\bbR^d \setminus \up_t$ is assigned to one of these clusters by some method for classification, the simplest one being by proximity (a point is assigned to the closest cluster). 

\section{\uppercase{Clustering via gradient lines: the gradient flow}}

To talk about gradient lines, we need to assume that the population density $f$ is differentiable. The gradient ascent line starting at a point $x$ is the curve given by the image of $\gamma_x$, the parameterized curve defined by the following ordinary differential equation (ODE)
\begin{equation} \label{gradient_flow}
\gamma_x(0) = x; \quad
\dot\gamma_x(t) = \nabla f(\gamma_x(t)).
\end{equation} 
By standard existence and uniqueness results for ODEs \citep[Ch 17]{hirsch2012differential}, if $\nabla f$ is locally Lipschitz, this curve exists and is unique, and it is defined on $[0,\infty)$ with $\gamma_x(t)$ converging to a critical point of $f$ as $t \to \infty$. 
See \figref{gradient_lines_2d} for an illustration.
Henceforth, we assume that $f$ is twice continuously differentiable, which is certainly enough for such gradient lines to exist.

\begin{figure}[!ht]
\centering
\includegraphics[scale=0.6]{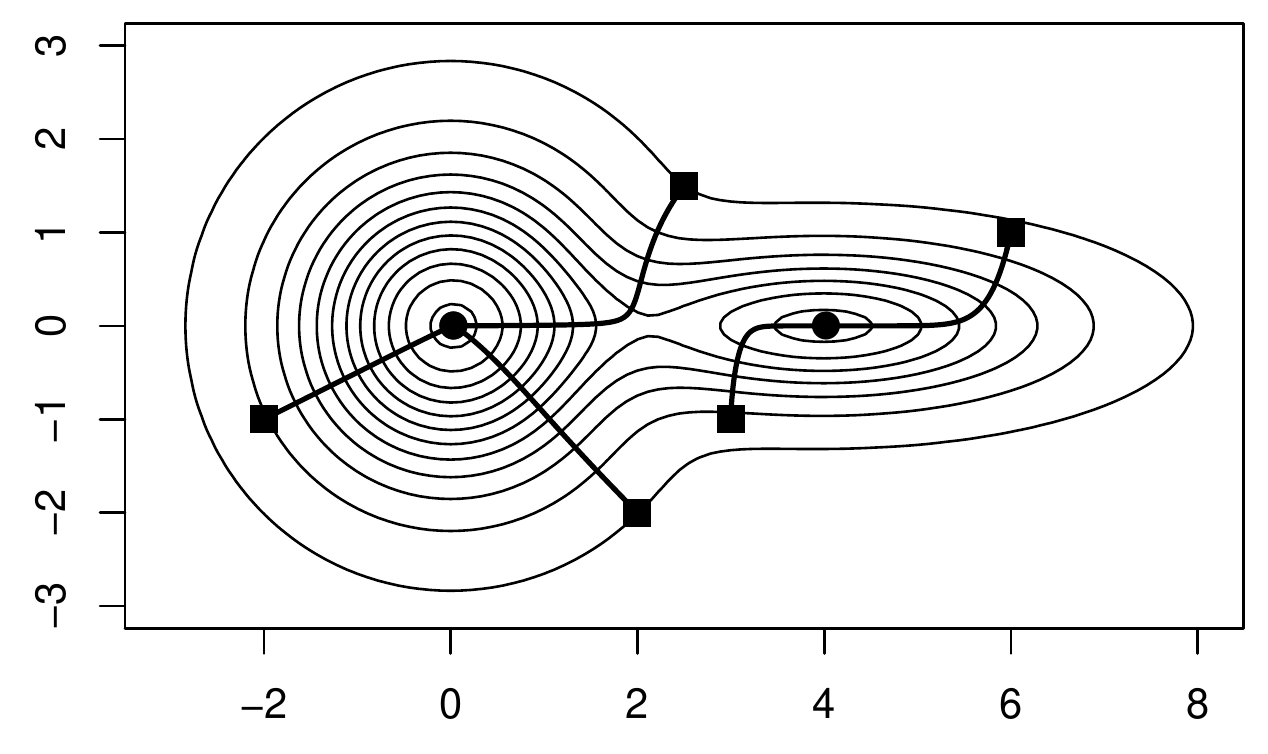}
\caption{A sample of gradient ascent lines for the density of \figref{level_sets_2d}. The square points are the starting points, while the round points are the end points, which are not only critical points but also local maxima (modes) in this example.} 
\label{fig:gradient_lines_2d}
\end{figure}

It is intuitive to define clusters based on local maxima, and \citet{fukunaga1975} suggest to ``assign each [point] to the nearest mode along the direction of the gradient" --- as opposed to the closest mode in Euclidean distance, for instance.
Define the basin of attraction of a critical point $x_0$ as $\{x: \gamma_x(\infty) = x_0\}$.
It turns out that, if $f$ is of Morse type \citep{milnor1963morse} inside its support, meaning that the Hessian of $f$ at any of its critical points is non-degenerate, then all these basins of attraction, sometimes called stable manifolds, provide a partition of the entire population.
In fact, the basin of attraction of the local maxima, by themselves, cover the population, except for a set of zero measure. 
For more background on Morse functions and their use in statistics, see the recent articles of \citet{chacon2015population} and \citet{ chen2017statistical}.

\begin{rem}
In their original article, \citet{fukunaga1975} also proposed an implementation: ``The approach uses the fact that the expected value of the observations within a small region about a point can be related to the density gradient at that point." 
The procedure is now known as the {\em blurring mean-shift algorithm} after \citet{cheng1995}, who suggested what is now known as the {\em mean-shift algorithm}, which is much closer to the plug-in approach suggested in our narrative. The mean-shift algorithm, and the twin problem of estimating the gradient lines of a density, are now well-understood \citep{cheng1995, comaniciu2002mean, cheng2004estimating, arias2016estimation, carreira2000mode}. 
The behavior of the blurring mean-shift algorithm is not as well understood, although some results do exist \citep{cheng1995, carreira2008generalised, carreira2006fast}.
\end{rem}

\section{\uppercase{Relating the cluster tree and the gradient flow}}

We assume that the density is twice continuously differentiable and of Morse type within its support so that we may freely discuss the partition of the population given by the gradient ascent lines. 

For the same population, consider the cluster tree. We saw that this object does not by itself provide a partition of the population, but a look at \figref{level_sets_2d} points to that possibility where, with a little imagination, we may see the contours (representing the level sets) as `moving' upward. As it happens, this can be formalized, and the result is a partition that coincides with that defined by the gradient flow.
Despite being quite intuitive, this correspondence appears to be novel. It is developed in the present section.

\subsection{The gradient flow follows the cluster tree}
\label{sec:flow_tree}
A close relationship between the cluster tree and the gradient flow can be anticipated by a look at \figref{gradient_lines_2d}, where it appears that gradient lines do not cross different clusters at the same level. This happens to be true, as we argue below.

To see the situation more clearly, take a point $x$ in the basin of attraction of a local maximum $x_*$ and let $\gamma_x$ denote the gradient line between $x$ and $x_*$ as defined in \eqref{gradient_flow}. 
Assume that $x \ne x_*$, so that the situation is not trivial.
We have assumed that $f$ is Morse, so that $x$ is (almost) generic.
The function $t \mapsto f(\gamma_x(t))$ is non-decreasing by construction and, as a consequence, $\gamma_x$ is `compatible' with the cluster tree in the sense that it does not cross from one connected component to another one at the same level. 
Indeed, for such a crossing would require an incursion in the region between the two connected components, say at level $t$, and that region is, by definition, outside the upper level set $\up_t$ and thus displays values of $f$ that are $\le t$. Crossing from one cluster to another at level $t$ would thus imply going from values of $f$ that are $>t$ when in the first cluster, to values $\le t$ when in the intermediary region, and finally to values $> t$ when entering the second cluster. 

To offer a different perspective, if $t := f(x)$, and $\cC$ is the component of $\up_t$ such that $x \in \cC$ --- so that $\cC$ is the `last' cluster containing $x$ when moving upward in the cluster tree (away from the root) --- then $\gamma_x(0) \in \cC$; $\gamma_x(s)$ belongs to a descendant of $\cC$ for all $s > 0$; and if $\cC_s$ denotes the last cluster containing $\gamma_x(s)$, then $\cC_r$ is a descendant of $\cC_s$ for all $r > s$. 

What was just said hinges on the fact that $\gamma_x$ does not pass through any critical point except at $\infty$ when reaching $x_*$, and in particular does not come into contact with any point where a level set splits. While the first part of this claim is well-known in dynamical systems theory, the second part is justified as follows.

\begin{prp} \label{prp:split}
Any point at the intersection of the closure of two connected components of $\up_t^\circ$ (the interior of $\up_t$) is a critical point. 
\end{prp}

Note that the statement is void unless $t$ is such that $\up_t$ has strictly fewer components than $\up_t^\circ$, meaning that a topological event happens at level $t$.

\subsection{The cluster tree follows the gradient flow}
\label{sec:tree_flow}
The cluster tree organizes the clusters (i.e., the connected components of the upper level sets) across all levels. We show now that the gradient flow provides a natural, almost canonical way of moving along the tree from the root (the population) to the leaves (the modes), thus reinforcing the case we are making that the cluster tree and the gradient flow are intimately related when it comes to defining a partition the entire population.

Suppose that there are no critical points at level $t$, i.e., no $x$ such that $f(x) = t$ and $\nabla f(x) = 0$. Note this is the generic situation, as the critical points form a discrete set by the fact that $f$ is Morse. Then, for $\eta > 0$ small enough, there are no critical points at any level between $t$ and $t+\eta$. By standard Morse theory \cite[Th 2.6]{nicolaescu2011invitation}, this implies that there are no `topological events' between levels $t$ and $t+\eta$ in that $\level_t$ and $\level_{t+\eta}$ are homeomorphic. In particular, these two level sets have the same number of connected components, and if $\cC_t$ is a connected component of $\level_t$, then it contains a single connected component, say $\cC_{t+\eta}$, of $\level_{t+\eta}$. Knowing all this, a natural way to move $\level_t$ to $\level_{t+\eta}$ is by (metric) projection of each component of $\level_t$ onto the component of $\level_{t+\eta}$ that it contains. It turns out that, by taking $\eta$ small enough but still positive, this operation becomes well-defined and a homeomorphism. 

\begin{thm} 
\label{thm:projection}
In the present context, for $\eta > 0$ small enough, the metric projection onto $\level_{t+\eta}$ is an homeomorphism sending $\level_t$ to $\level_{t+\eta}$. 
\end{thm}

\begin{rem}
Henceforth, the projection of $\level_t$ onto $\level_{t+\eta}$ will be denoted $P_{t, \eta}$. The previous proposition states that $P_{t,\eta}$ is a homeomorphism when there are no critical points at level $t$ and $\eta$ is small enough, but in fact $P_{t,\eta}$ is a diffeomorphism under the same circumstances. 
See \thmref{diffeomorphism} in \secref{finer}. 
\end{rem}

The metric projection $P_{t, \eta}$ has been defined when $\cL_t$ is regular level set and $\eta$ is small enough. Even when $t$ is a critical value of $f$, meaning that $\cL_t$ contains a critical point, for any non-critical point $x\in\cL_t$, the projection $P_{t, \eta}(x)$ is still well defined, if $\eta$ is small enough.

\begin{prp}
\label{prp:singleton}
For any $x \in \level_t$ such that $\|\nabla f(x)\|\neq 0$, $P_{t, \eta}(x)$ is a singleton for $\eta$ small enough.
\end{prp} 

This metric projection is in fact related to gradient ascent in the following way: in the infinitesimal regime where $\eta \to 0$, the transformation coincides with the gradient flow at a certain speed specified below. 

\begin{prp} \label{prp:projection_infinitesimal}
For any $x \in \level_t$ such that $\|\nabla f(x)\|\neq 0$,
\[\frac{P_{t,\eta}(x) - x}\eta \to \frac{\nabla f(x)}{\|\nabla f(x)\|^2}, \quad \eta \to 0.\]
\end{prp}

%

For a more global (rather than local) relationship between the level sets and the gradient lines, consider the following gradient flow
\begin{equation} \label{gradient_flow_norm}
\begin{gathered}
\zeta_x(0) = x; \;\;\;
\dot\zeta_x(t) = 
\begin{cases}
\frac{\nabla f(\zeta_x(t))}{\|\nabla f(\zeta_x(t))\|^2}, & \text{if } \nabla f(\zeta_x(t)) \ne 0; \\
0, & \text{otherwise}.
\end{cases}
\end{gathered}
\end{equation} 
Starting at the same point $x$, the gradient line defined by $\zeta_x$ is the same as the gradient line defined by $\gamma_x$ given in \eqref{gradient_flow}, but it is traveled at a different speed.

If the metric projection is applied iteratively, then its local approximation to the gradient ascent, as revealed in \prpref{projection_infinitesimal}, can be extended to a global approximation to the gradient flow in \eqref{gradient_flow_norm}. More precisely, starting from a regular point $x_0=x$, we can define a sequence $x_0,x_1,\dots,$ such that $x_{j+1} = P_{t_j,\eta}(x_j)$, where $t_j =f(x_j)$, and intuitively, the polygonal line connecting the sequence approximates the gradient flow $\zeta_{x}$. The consistency of this approximation is rigorously analyzed in our subsequent paper \citep{arias2021asymptotic}. For $\eta$ small enough, the metric projection can be iteratively applied if the trajectory of $\zeta_{x}$ stays away from critical points, until it enters a leaf node of the cluster tree and gets close to a local mode. The metric projections thus push almost all the points to the same leaf nodes as the corresponding gradient flow does. As $\eta\to0$, this operation on the cluster tree coincides with clustering based on the gradient flow.

The gradient flow defined in \eqref{gradient_flow} is the vector field corresponding to $F(x) = \nabla f(x)$, while the variant of \eqref{gradient_flow_norm} corresponds to $F(x) = \nabla f(x)/\|\nabla f(x)\|^2$.
\citet{fukunaga1975} proposed to use a different variant: the one based on $F(x) = \nabla f(x)/f(x)$, the idea being to quickly move points in low density regions to higher density regions. 
When relating the gradient lines to the level sets, this variant of the gradient flow is particularly compelling because it is, in effect, parameterized by the level.

\begin{thm}\label{thm:zeta}
The gradient flow given in \eqref{gradient_flow_norm} has the following property. Starting from a point $x$ at level $t > 0$, meaning that $f(x)=t$, it holds that $f(\zeta_x(s-t)) = s$ for all $t < s < \max f$ as long as $\nabla f(y)\neq 0$ for all $y\in\{\zeta_x(\tau):\tau\in[0,s]\}$. 
In fact, the transformation $x \mapsto \zeta_x(s-t)$ provides a homeomorphism from $\level_t$ to $\level_s$ whenever there are no critical points at any level anywhere between $t$ and $s$, inclusive.
\end{thm}

\begin{rem}
It turns out that the same transformation is a diffeomorphism under the same circumstances. 
See \thmref{diffeomorphism2} in \secref{finer}. 
\end{rem}

\paragraph{Level set methods} 
There is another interesting connection between the cluster tree and partial differential equations that offers a different perspective on the gradient flow \eqref{gradient_flow_norm}. We presented this flow as a natural way of moving along the cluster tree, from root to leaves, which is how we understand the cluster tree as enabling a partitioning of the entire space. We say this with full awareness that our results are limited in that we only show in \thmref{zeta} that this gradient flow moves level sets between which there are no topological changes. 

Moving past a topological change presents a real challenge, and it turns out that this problem has been addressed in the PDEs literature, where researchers were confronted with such challenges when moving `fronts' (i.e., surfaces of co-dimension~1, for examples, curves in the plane) according to some motion, most notably some form of motion under curvature. \cite{osher1988fronts} pioneered an approach which consists in representing the moving front, say $(\Gamma_t : t \ge 0)$, as the zero level set of a time-varying function, say $\phi$, so that $\Gamma_t = \{x : \phi(x, t) = 0\}$. The evolution of the moving front is then implemented via an evolution of the representing function. This approach has lead to what is nowadays appropriately known as {\em level set methods} \citep{osher2003level, sethian1999level}.

To see how this method is applied, suppose we want to evolve a simple closed curve $\Gamma$ along its normal inward direction at velocity $v$. (The function $v$ is defined on the entire space where the evolution takes place.)
Take a function $\phi$ such that $\Gamma = \{x : \phi(x,0) = 0\}$ and then evolve $\phi$ as follows
\begin{align}
\dot\phi + v \|\nabla\phi\| = 0.
\end{align}
Then $\Gamma_t := \{x : \phi(x,t) = 0\}$ evolves $\Gamma_0 = \Gamma$ as prescribed.

Coming back to our setting, choose $\phi(x,t) = f(x) -t$, which provides indeed a representation of $\level_t$ for all $t$. Note that the movement induced by $\phi$ is exceedingly simple, as it consists of moving the landscape given by $f$ vertically at constant speed.  And it turns out that, since $\dot\phi = -1$ and $\nabla\phi = \nabla f$, this movement is exactly the one induced by $\zeta$, at least in between critical levels. (To be sure, the normal to $\level_t$ at some $x$ is $N(x) = \nabla f(x)/\|\nabla f(x)\|$, so that $\zeta$ is given by $F(x) = v(x) N(x)$ with $v(x) := 1/\|\nabla f(x)\|$ standing for the velocity.)

\section{\uppercase{Discussion}}
\label{sec:discussion}

In this paper we have established what we regard as an interesting and perhaps overlooked correspondence between level set clustering and gradient line clustering, which brings these two views on clustering quite close together, in our view.
In view of this correspondence, it is quite natural to discuss both approaches together, perhaps under the name of {\em modal clustering}, a term that has been already used to refer to a class of related methods \citep{menardi2016review, chacon2020modal}, which includes single-linkage clustering \cite{hartigan1981consistency, hartigan1977distribution, hartigan1977clusters, penrose1995single} as well as some methods based on nearest neighbors \citep{stuetzle2003estimating, stuetzle2010generalized} such as DBSCAN \citep{ester1996density} and even some forms of spectral clustering \citep{arias2011clustering}.

The correspondence we crafted between the cluster tree and the gradient ascent flow could, however, be strengthened, for example by showing that the transition over level sets that contain a saddle point does not fundamentally change anything. Like the multi-pronged connection that we describe in this paper --- in particular in the form of \thmref{projection} and \thmref{zeta} --- it seems intuitively clear that the `crossing' of such a level set does not affect things in a substantial way. But this remains to be established with mathematical rigor. In any case, the message that we conveyed in \secref{flow_tree} remains valid: clustering by gradient ascent is compatible with the cluster tree. 
In recent work \cite{arias2021asymptotic}, we take a different approach --- alluded to earlier --- which allows us to draw another strong correspondence between level set clustering and gradient line clustering while, this time, completely avoiding the issue of dealing with change in topology of level sets.

Although we draw a correspondence between these two views on clustering --- in terms of how they define the task of clustering itself --- they obviously remain distinct, and may inform methodology in different ways. For example, they can be combined to group what are deemed inessential modes, and subsequently, the associated clusters. We can imagine the following algorithm, which offers a principled way based on the gradient flow to deal with the region outside the chosen upper level set --- what \citet{stuetzle2010generalized} call the `fluff'. 
Given a threshold $t>0$,
\begin{enumerate}\itemsep0em
\item Compute the upper level sets at level $t$.
\item Partition the population according to the basins of attraction of the modes. 
\begin{enumerate}\itemsep0em
\item If two modes are in the same upper $t$-level set cluster, merge their basins.
\item If a mode is not in any upper $t$-level set cluster, then its basin is considered `noise'.
\end{enumerate} 
\end{enumerate}
The threshold $t$ can be chosen as it would in a purely level set based clustering algorithm. 

\newpage
\section{\uppercase{Proofs}}
\label{sec:proofs}

\subsection{Preliminaries}

For a point $x$ and $\delta > 0$, $\ball(x,\delta)$ is the open ball centered at $x$ of radius $\delta$.
For $\cA \subset \bbR^d$ and $\delta > 0$, define its $\delta$-neighborhood as 
\begin{equation}
\ball(\cA, \delta) 
:= \bigcup_{a \in \cA} \ball(a, \delta)
= \{x : \dist(x, \cV) < \delta\}, 
\end{equation}
where $\dist(x, \cA) := \inf\{\|x-a\| : a \in \cA\}$. Denote the closure and interior of $\cA$ by $\bar{\cA}$ and $\cA^\circ$, respectively. 
The Hausdorff distance between $\cA, \cB \subset \bbR^d$ is defined as
\begin{equation} \label{hausdorff}
d_H(\cA, \cB) 
:= \max\big\{d_H(\cA \mid \cB), d_H(\cB \mid \cA)\big\},
\end{equation}
where
\begin{equation}
d_H(\cA \mid \cB)
:= \sup_{a\in \cA} \inf_{b \in \cB} \|a-b\|.
\end{equation}
The (metric) projection of a point $x$ onto a closed set $\cA$ is the subset points in $\cA$ that are closest to $x$. (That subset is nonempty if $\cA$ is non-empty.) 
We say that $x$ has a unique projection onto $\cA$ if its projection is a singleton. 
The reach of $\cA \subset \bbR^d$, ${\rm reach}(\cA)$, is defined as the supremum over $\delta>0$ such that every point in $\ball(\cA, \delta)$ has a unique projection onto $\cA$ \citep{federer1959curvature}. Many things are known about the reach of a set. We will use the following properties.

The following is an immediate consequence of \citep[Lem 4.8(2)]{federer1959curvature}.
\begin{lem} \label{lem:federer4.8(2)}
Suppose $\cA \subset \bbR^d$ is a $(d-1)$-dimensional differentiable manifold with positive reach $r > 0$. Let $P$ stand for the metric projection onto $\cA$, which is well-defined as a function on the $r$-neighborhood of $\cA$. Then for any $x \in \ball(\cA, r)$, $P(x)-x$ is orthogonal to the tangent space of $\cA$ at $P(x)$. 
\end{lem}

The following is an immediate consequence of \citep[Lem 4.11]{federer1959curvature}.
\begin{lem} \label{lem:federer4.11}
Suppose the density $f$ is twice continuously differentiable. For any $t>0$, if there are no critical point on $\level_t$ defined in \eqref{Lt}, then ${\rm reach}(\level_t) > 0$. 
\end{lem}

\subsection{Proof of \prpref{split}}
Let $x_0$ be a regular point in $\level_t$. There exists a bounded open set $\cV\subset\bbR^d$ containing $x_0$ such that $\level_t\cap \cV$ is a $(d-1)$-dimensional surface by the Constant Rank Level Set Theorem \citep[Th 5.12]{lee2013}. 
If necessary, shrink $\cV$ so that $p(x):=\|\nabla f(x)\|>\epsilon_0$ for some $\epsilon_0>0$ for all $x\in \level_t\cap \cV$. \lemref{federer4.11} can be used to show that the set $\level_t\cap \cV$ has a positive reach, denoted by $a_0$. Denote $N(x) = \nabla f(x) / \|\nabla f(x)\|,$ which is a unit normal vector of $\level_t\cap \cV$ at $x$. For a fixed $\delta_0\in(0,a_0),$ define $\cW(\delta_0)=\{x+u N(x): x\in \level_t\cap \cV, u\in[-\delta_0,\delta_0]\}$, and $b_0 := \max\big\{\|\nabla^2 f(y)\| : y \in \cW(\delta_0)\big\}$. Note that $b_0 < \infty$ by the fact that $\nabla^2 f$ is continuous on $\bbR^d$ and $\cW(\delta_0)$ is bounded.
A Taylor expansion from $x\in\level_t\cap \cV$ in the direction of $N(x)$ gives
\begin{equation}
\label{Directional_Taylor}
\big| f(x + u N(x)) - f(x) - u p(x)\big| \le \tfrac12 b_0 u^2,  
\end{equation}
for all $u \in [-\delta_0,\delta_0]$. 
Thus, for $u \in [0, (\epsilon_0/b_0)\wedge \delta_0]$, we have
\begin{equation}
f(x + u N(x)) 
\ge f(x) + u \epsilon_0 - \tfrac12 b_0 u^2
\ge t + \tfrac12 u \epsilon_0.
\end{equation}
Similarly, for $u \in [0, (\epsilon_0/b_0)\wedge \delta_0]$,
\begin{equation}
f(x - u N(x)) 
\le f(x) - u \epsilon_0 + \tfrac12 b_0 u^2
\le t - \tfrac12 u \epsilon_0.
\end{equation}
In other words, in a small neighborhood the two sides of the surface $\level_t\cap \cV$ have $f$ values strictly below and above $t$, respectively. This then implies that $x_0$ cannot be at the intersection of the closure of two connected components of $\up_t^\circ$.

\subsection{Proof of \thmref{projection}}

By applying \cite[Th 1]{chazal2007normal}, it suffices to show that, for $\eta > 0$ small enough,
\begin{equation}
\label{normalcomp}
d_H(\level_t, \level_{t+\eta}) < c \min({\rm reach}(\level_{t}), {\rm reach}(\level_{t+\eta})),
\end{equation}
with $c := 2-\sqrt{2}$, which guarantees that $\level_t$ and $\level_{t+\eta}$ are `normal compatible', meaning that both $P_{t, \eta}$ and $P_{t+\eta, -\eta}$ are homeomorphisms.

First, it follows from \lemref{federer4.11} that there exists $\eta_0>0$ such that 
\begin{equation}
\min({\rm reach}(\level_{t}), {\rm reach}(\level_{t+\eta}))>0, \quad \forall \eta\in[0,\eta_0].
\end{equation} 

Therefore \eqref{normalcomp} holds for $\eta > 0$ small enough if it is the case that $d_H(\level_t, \level_{t+\eta}) \to 0$ as $\eta\to0$. 
We establish this convergence in what follows.

{\em Claim: $d_H(\level_t \mid \level_{t+\eta}) \to 0$ as $\eta\to0$.} 
Recall that $p(x) = \|\nabla f(x)\|$ and define $p_0 := \min\{p(x) : x \in \level_t\}$. Note that $p_0 > 0$ by the fact that $\nabla f(x) \ne 0$ for all $x \in \level_t$, that $\nabla f$ is continuous, and that $\level_t$ is compact.
For a fixed $\delta_0>0,$ define $c_0 := \max\big\{\|\nabla^2 f(y)\| : y \in \bar\ball(\level_t, \delta_0)\big\}$, and note that $c_0 < \infty$ by the fact that $\nabla^2 f$ is continuous and $\bar\ball(\level_t, \delta_0)$ is compact.
Similar to \eqref{Directional_Taylor}, a Taylor expansion gives
\begin{equation}
\big| f(x + u N(x)) - f(x) - u p(x)\big| \le \tfrac12 c_0 u^2,  
\end{equation}
for all $u \in [-\delta_0,\delta_0]$. 
Applying this to $x \in \level_t$ and $u \in [0, (p_0/c_0)\wedge \delta_0]$ yields
\begin{equation}
f(x + u N(x)) 
\ge f(x) + u p_0 - \tfrac12 c_0 u^2
\ge t + \tfrac12 u p_0.
\end{equation}
In particular, when $\eta$ is small enough such that $2\eta/p_0 \le (p_0/c_0)\wedge \delta_0$, we have $f(x + (2\eta/p_0) N(x)) \ge t + \eta$, and by continuity of $f$ this implies that there is $u \in [0,2\eta/p_0]$ such that $f(x+uN(x)) = t+\eta$, i.e., $x+uN(x) \in \level_{t+\eta}$, forcing $\dist(x, \level_{t+\eta}) \le 2\eta/p_0$. This being uniform in $x \in \level_t$, since it is valid for $\eta \le p_0^2/2c_0$, we have that $d_H(\level_t \mid \level_{t+\eta}) \le 2\eta/p_0$, valid for $\eta$ small enough, proving the claim.

{\em Claim: $d_H(\level_{t+\eta} \mid \level_t) \to 0$ as $\eta\to0$.}
This claim can be established in a way that is completely analogous to one that lead to the previous claim.

\subsection{Proof of \prpref{singleton}}
\begin{proof}
The reach of $\cA \subset \bbR^d$ at $y\in\cA$, ${\rm reach}(\cA,y)$, is defined as the supremum over $\delta>0$ such that every point in $\ball(y, \delta)$ has a unique projection onto $\cA$ \citep{federer1959curvature}. For $r>0$, define
\begin{equation}
\pi_y(r):=\frac{\inf_{z\in\ball(y,r)}\|\nabla f(z)\|}{\sup_{z\in\ball(y,2r)}\|\nabla^2 f(z)\|}.
\end{equation}
Note that $\pi_y$ is a decreasing function of $r$. By \citep[Lem 4.11]{federer1959curvature}, for every $y\in\bbR^d$ with $\|\nabla f(y)\|\ne0$, we have
\begin{equation}
\label{reach_bound}
{\rm reach}(\cL_{f(y)},y) \ge \min\{r/2, \pi_y(r) \}, 
\end{equation}
for all $r>0$ whenever $\inf_{z\in\ball(y,r)}\|\nabla f(z)\|>0.$
 Since $\|\nabla f(x)\|>0$, there exists $\delta_0>0$ such that $\inf_{y\in\ball(x,\delta_0)}\|f(y)\|>0$, by the continuity of the gradient. Define 
\begin{equation}
\delta_1= \min\{\delta_0/4, \pi_x(\delta_0) \}>0, \quad\rm{ and } \quad\delta_2 = \delta_0/2+\delta_1\in(\delta_1,\delta_0).
\end{equation}
Note that for any $y\in\ball(x,\delta_1)$, $\pi_y(\delta_2) \ge \pi_x(\delta_0)$ because $\ball(y,\delta_2)\subset \ball(x,\delta_0)$ and $\ball(y,2\delta_2)\subset \ball(x,2\delta_0)$. Then it follows from \eqref{reach_bound} that
\begin{equation}
{\rm reach}(\cL_{f(y)},y) \ge \min\{\delta_2/2, \pi_y(\delta_2) \} \ge \delta_1.
\end{equation}
By the definition of reach, the projection from $x$ to $\cL_{f(y)}$ is unique for all $y\in\ball(x,\delta_1)$. Consequently, for every positive $\eta < \sup_{y\in \ball(x,\delta_1)} - f(x)$, the projection from $x$ to $\cL_{t+\eta}$ is unique.
\end{proof}

\subsection{Proof of \prpref{projection_infinitesimal}}

\begin{proof}
Here $x$ and $t$ are considered fixed. 
Let $P_\eta$ be short for $P_{t,\eta}$. 
It is well-known that for any non-critical point $y\in \level_s$, $\nabla f(y)$ is orthogonal to $\level_s$ and pointing inwards. When $\eta$ is small enough, $P_\eta(x) -x$ is parallel to $\nabla f(P_\eta(x))$ by \lemref{federer4.8(2)}. So we can write 
\begin{equation}
\label{Peta1}
P_\eta(x) -x = c_\eta \nabla f(P_\eta(x)),
\end{equation}
for some scalar $c_\eta$. Note that $c_\eta \to 0$ as $\eta \to 0$, since $\nabla f(x) \ne 0$. 
Using a Taylor expansion, we have for some $s_\eta\in(0,1)$,
\begin{align}\label{Peta2}
\eta 
&= f(P_\eta(x)) - f(x) \\
&= [P_\eta(x) - x]^\top \nabla f(s_\eta P_\eta(x) + (1-s_\eta) x) \nonumber\\
& = c_\eta [\nabla f(P_\eta(x))]^\top \nabla f(s_\eta P_\eta(x) + (1-s_\eta) x).
\end{align}
Extracting the expression of $c_\eta$ from \eqref{Peta2} and plugging it into \eqref{Peta1}, we get
\begin{align}
\frac{P_\eta(x) - x}\eta = \frac{\nabla f(P_\eta(x))}{[\nabla f(P_\eta(x))]^\top \nabla f(s_\eta P_\eta(x) + (1-s_\eta) x)}.
\end{align}
We then conclude by noting that, as $\eta \to 0$, $P_\eta(x) \to x$.
This implies that $\nabla f(P_\eta(x)) \to \nabla f(x)$, by the fact that $f$ is twice continuously differentiable, so that $\nabla f$ is continuous. 
And this also implies that $\nabla f(sP_\eta(x) + (1-s) x) \to \nabla f(x)$, because of the uniform continuity of $\nabla f$ on the line segment $[x, P_\eta(x)]$. 
\end{proof}

\subsection{Proof of \thmref{zeta}}
\begin{proof}
The result is essentially known, but since we were not able to locate it in this form, we provide a short proof.
Throughout, $t$ and $s$ are considered fixed.

Using elementary calculus and the definition in \eqref{gradient_flow_norm}, we have
\begin{align}
\label{level_parameterization}
f(\zeta_x(s-t)) - f(x) 
&= \int_0^{s-t} \frac{\d f(\zeta_x(\tau)) }{ \d \tau} \d \tau \nonumber\\
&= \int_0^{s-t} [\nabla f(\zeta_x(\tau))]^\top\, \dot\zeta_x(\tau) \d\tau \nonumber\\
&= \int_0^{s-t} 1\, \d\tau
= s -t.
\end{align}
Hence, $f(\zeta_x(s-t)) = f(x) + s-t = s$, giving the first part of the statement. 

Let $\psi(x) := \zeta_x(s-t)$. We show that $\psi:\level_t \to \level_s$ is a homeomorphism when there are no critical points in $\level_u$ for any $u\in[t,s]$, or equivalently, no critical point in $\cV := \bigcup_{u\in[t,s]} \level_u = \up_t \setminus \up_s^\circ$.
Under this condition, the trajectories of $\{\zeta_{x}(\tau): \tau\in[0,s-t]\}$ and $\{\zeta_{y}(\tau): \tau\in[0,s-t]\}$ do not intersect for any two different starting points $x,y\in\level_t$, as is well-known. This implies that the map $\psi:\level_t \to \level_s$ is injective. 
For any $z\in\level_s$, consider the gradient {\em descent} flow ($\zeta$ in reverse) given by
\begin{equation} \label{gradient_flow_norm_reverse}
\begin{gathered}
\zeta^\downarrow_z(0) = z; \;\;\;
\dot\zeta^\downarrow_z(t) = 
\begin{cases}
-\frac{\nabla f(\zeta^\downarrow_z(t))}{\|\nabla f(\zeta^\downarrow_z(t))\|^2}, & \text{if } \nabla f(\zeta^\downarrow_z(t)) \ne 0; \\
0, & \text{otherwise}.
\end{cases}
\end{gathered}
\end{equation} 
Let $x = \psi^\downarrow(z) := \zeta^\downarrow_z(s-t)$. It follows from a calculation similar to \eqref{level_parameterization} that $x\in\level_t$. Notice that $z=\zeta_{x}(s-t)=\psi(x)$, so that we can write $x=\psi^{-1}(z)$, and therefore the map $\psi:\level_t \to \level_s$ is surjective. In the process, we have shown that the inverse of the gradient ascent operation --- the function $\psi$ defined via $\zeta$ in \eqref{gradient_flow_norm} --- is given by the corresponding gradient descent operation --- the function $\psi^\downarrow$ defined via $\zeta^\downarrow$ in \eqref{gradient_flow_norm_reverse}.

The proof is completed after we show that $\psi$ is continuous. (The continuity of $\psi^{-1} = \psi^\downarrow$ can be proved in a similar way.)
We start by stating that, because $\nabla f(x) \ne 0$ for all $x \in \cV$, and by the assumption that $\nabla f$ is continuous, there is $\delta>0$ such that $\|\nabla f(x)\|>0$ for all $x\in \cV_\delta := \bar\ball(\cV, \delta)$.
Define $F :=\nabla f/\|\nabla f\|^2$, which is the map that drives the gradient flow $\zeta$ in \eqref{gradient_flow_norm}.
This map is continuous, and $\cV_\delta$ is compact, so that there is $c_0$ such that $\|F(x)\| \le c_0$ for all $x \in \cV_\delta$.
The map is also differentiable, with
\begin{equation}
\nabla F = \frac{\nabla^2 f}{\|\nabla f\|^2} - \frac{2\nabla f \nabla f^\top \nabla^2 f}{\|\nabla f\|^4},
\end{equation}
where $\nabla F$ denotes here the Jacobian matrix of $F$, while $\nabla^2 f$ denotes the Hessian of $f$. And by continuity of $\nabla f$ and of $\nabla^2 f$, and the fact that $\|\nabla f(x)\|>0$ for all $x\in \cV_\delta$, and again the fact that $V_\delta$ is compact, there exists $c_1$ such that $\|\nabla F(x)\| \leq c_1$ for all $x\in\cV_\delta$, where $\|\cdot\|$ denotes any matrix norm. In other words, $F$ is bounded and has bounded gradient on $\cV_\delta$.

Define $\cW := \ball(\cV, \delta/2)$.
Note that $\cW$ is open with $\cW \subset \cV_\delta$, so that $\|F(x)\| \le c_0$ and $\|\nabla F(x)\| \leq c_1$ for all $x \in \cW$.
If $x, y \in \cW$ are such that $\|x-y\| \le \delta/4$, there must be $z\in\cV$ such that $x, y \in \ball(z, \delta/2)$, and because that ball is convex, we have
\begin{equation}
\|F(x) - F(y)\| 
\le c_1 \|x-y\|.
\end{equation}
If, on the other hand, $\|x-y\| > \delta/4$, then we can simply write
\begin{equation}
\|F(x) - F(y)\| 
\le 2 c_0
= \frac{2 c_0}{\delta/4} \delta/4
\le \frac{8 c_0}{\delta} \|x-y\|.
\end{equation}
Hence, $F$ is Lipschitz on $\cW$ with corresponding constant bounded by $C := \max\{c_1, 8c_0/\delta\}$.
Finally, note that $\cV \subset \cW$ and that for any $x \in \level_t$, $\zeta_x(\tau) \in \cV$ for all $\tau \in [0,s-t]$, since $f(\zeta_x(\tau)) = t + \tau \in [t, s]$.
All together, we are in a position to apply a standard result on the dependence of the gradient flow on the initial condition, for example, the main theorem in \citep[Sec 17.3]{hirsch2012differential}, which gives the bound
\begin{equation}
\begin{gathered}
\|\zeta_x(\tau) - \zeta_y(\tau)\|
\le \|x-y\| \exp(C \tau), \quad
\forall \tau \in [0,s-t], \quad \forall x,y \in \level_t,
\end{gathered}
\end{equation}
which, in particular, implies
\begin{equation}
\label{psi_diff_bound}
\begin{gathered}
\|\psi(x) - \psi(y)\|
\le C_* \|x-y\|, \quad
\forall x,y \in \level_t,
\end{gathered}
\end{equation}
with $C_* := \exp(C (s-t))$.
\end{proof}

\section{\uppercase{Finer results}}
\label{sec:finer}

In this section we state and prove stronger versions of our main results, namely, \thmref{projection} and \thmref{zeta}. In these results, we prove that transformation, respectively, a metric projection onto a close enough level set and a gradient flow, is a homeomorphism. While this was enough for our purposes, below we prove that, in fact, the same transformations are diffeomorphisms.

\begin{thm} 
\label{thm:diffeomorphism}
In the present context, for $\eta > 0$ small enough, the metric projection onto $\level_{t+\eta}$ is a diffeomorphism sending $\level_t$ to $\level_{t+\eta}$. 
\end{thm}

\begin{proof}
In \thmref{projection} we have shown that the metric projection is an homeomorphism for $\eta$ small enough. Consider its inverse, denoted as $Q$, which is a homeomorphism sending $\level_{t+\eta}$ to $\level_t$. By definition, for any $x\in\level_{t+\eta}$, 
\begin{equation}
\label{Q_def}
Q(x) = x- \tau(x) N(x) ,
\end{equation} 
where $N(x) := \nabla f(x) / \|\nabla f(x)\|$ is a unit normal vector of the surface $\level_{t+\eta}$ at $x$, while $\tau(x) := \inf\{s>0: f(x - sN(x)) = t\}$. 
The fact that the projection is a homeomorphism, as established in \thmref{projection}, guarantees that for $\eta$ small enough, there exists $\beta\in(0,\eta)$ such that the definition of $Q$ can be naturally extended (in the same way as in \eqref{Q_def}) to the set 
\begin{equation}
\cG:= \bigcup_{u\in(t+\eta - \beta,t+\eta+\beta)} \level_u = \up_{t+\eta-\beta}^\circ \setminus \up_{t+\eta + \beta},
\end{equation}
such that $\inf_{x\in \cG}\|\nabla f(x)\|>c_0$ for some $c_0>0$. We still denote this extension by $Q$.

We show that $Q$ is continuously differentiable on $\cG$. The problem is reduced to the continuous differentiability of $\tau$, because \[\nabla N = \|\nabla f\|^{-1}[\I - NN^\top]\nabla^2 f,\]
which is continuous under the assumption that $f$ is twice continuously differentiable. For any fixed point $x\in\cG$ and any $s\in\bbR$, denote $g_x(s) = f(x - sN(x))$. Then
\begin{align}
\label{dot_gx}
\dot g_x(s) & = - N(x)^\top \nabla f(x - sN(x)) \nonumber \\
& = - N(x)^\top \Big[\nabla f(x) - s \int_0^1 \nabla^2 f(x - usN(x)) \d u \; N(x) \Big] \nonumber\\
& = - \|\nabla f(x)\| + s N(x)^\top \Big[\int_0^1 \nabla^2 f(x - usN(x)) \d u \Big] N(x).
\end{align} 
Hence under the assumption that $f$ is twice continuously differentiable, there exists $s_0>0$ such that 
\begin{equation}
\label{dot_gx_bound}
\dot g_x(s) < - \frac{1}{2}\|\nabla f(x)\| < -\frac{1}{2}c_0<0,
\end{equation}
for all $s\in[0,s_0]$, which means that $g_x$ is strictly decreasing in $[0,s_0]$ and $g_x^{-1}$ exists on $(g_x(s_0), g_x(0))$. Furthermore, using a Taylor expansion, we have that, for $x\in\level_{t+\alpha}$ with $\alpha\in(\eta-\beta,\eta+\beta)\subset(0,2\eta)$, and for $s\in[0,s_0]$,
\begin{align}
g_x(s) = g_x(0) + s \dot g_x(\tilde s) < (t+\alpha) -\frac{s}{2} \|\nabla f(x)\|, 
\end{align}
where $\tilde s\in[0,s] \subset [0,s_0]$. If $\eta < \frac{1}{4} s_0 c_0$, then 
\begin{equation}
\label{gxs_bound}
g_x(s_0) < t + 2\eta - \frac{s_0}{2} c_0 < t.
\end{equation}
Therefore $t\in (g_x(s_0), g_x(0))$ for $\eta$ small enough, and we can write
\begin{equation}
\tau(x) = g_x^{-1}(t).
\end{equation}
It also follows from \eqref{gxs_bound} that for $x\in\cG$,
\begin{equation}
\label{tau_bound}
\tau(x) < s_0 < \frac{4\eta}{c_0}.
\end{equation}

Below we show that $\tau$ is continuous on $\cG$. For $x\in\cG$, which is an open subset of $\bbR^d$, consider $\delta\in\bbR^d$ with $\|\delta\|$ small enough such that $B(x,\|\delta\|) \subset \cG$. By the definition of $\tau$, 
\begin{align}
f(x - \tau(x) N(x)) = t = f(x+\delta - \tau(x+\delta) N(x+\delta)).
\end{align}
Applying the mean value theorem on the two sides of the above equation, we get
\begin{align}
& f(x) -  \tau(x) N(x)^\top \int_0^1 \nabla f(x - u\tau(x)N(x))\d u \\
&= f(x+\delta) -  \tau(x+\delta) N(x+\delta)^\top \int_0^1 \nabla f(x + \delta - u \tau(x+\delta)N(x+\delta)) \d u.
\end{align}
Rearranging the terms above and using a telescoping argument, we have
\begin{align}
\label{telescoping}
& f(x) - f(x+\delta) \nonumber\\
&= [ \tau(x) -  \tau(x+\delta)] N(x)^\top \int_0^1 \nabla f(x - u\tau(x)N(x))\d u \nonumber\\
& \quad +   \tau(x+\delta) [ N(x) - N(x + \delta)]^\top \int_0^1 \nabla f(x - u\tau(x)N(x))\d u \nonumber\\
& \quad +  \tau(x+\delta) N(x + \delta)^\top \int_0^1 [\nabla f(x - u\tau(x)N(x)) - \nabla f(x - u \tau(x+\delta)N(x)) ]\d u \nonumber\\
& \quad +  \tau(x+\delta) N(x + \delta)^\top \int_0^1 [\nabla f(x - u\tau(x+\delta)N(x)) - \nabla f(x - u \tau(x+\delta)N(x+\delta)) ]\d u \nonumber\\
& \quad +  \tau(x+\delta) N(x + \delta)^\top \int_0^1 [\nabla f(x - u \tau(x+\delta)N(x+\delta)) - \nabla f(x + \delta - u \tau(x+\delta)N(x+\delta)) ]\d u \nonumber\\
&=: I_1 + I_2 + I_3 +I_4 + I_5.
\end{align} 
Recalling the expression of $\dot g_x$ in \eqref{dot_gx}, we can write for some $\tilde s(x)\in (0,\tau(x))$, 
\begin{align}
I_1 = - \dot g_x(\tilde s(x)) [\tau(x) - \tau(x+\delta)].
\end{align}
To study $I_2, I_3,I_4$, and $I_5$, we use further Taylor expansions. (Remember that $f$ is assumed twice continuously differentiable.) 
Doing that leads to
\begin{align}
I_2 & = - \tau(x+\delta) \delta^\top \Big[\nabla N(x)^\top \int_0^1 \nabla f(x - u\tau(x)N(x))\d u \Big] + o(\|\delta\|) \\
& =: -\tau(x+\delta) \delta^\top r_1(x) + o(\|\delta\|).
\end{align}
Similarly, 
\begin{align}
I_3 & = - \tau(x+\delta) [\tau(x) - \tau(x+\delta)] \\
&\hspace{1cm} \times N(x + \delta)^\top \Big[\int_0^1 \int_0^1 u\nabla^2 f(x - u(w\tau(x) + (1-w) \tau(x+\delta))N(x)) \d u \d w \Big] N(x),\\
I_4 & = \tau^2 (x+\delta) \delta^\top \nabla N(x)^\top\\
&\hspace{1cm}  \Big[\int_0^1 \int_0^1 u \nabla^2 f(x - u\tau(x+\delta)(wN(x) + (1-w)N(x+\delta))) \d u \d w \Big] N(x + \delta) + o(\|\delta\|),\\
I_5 & = - \tau(x+\delta) \delta^\top \Big[\int_0^1 \nabla^2 f(x - u\tau(x+\delta)N(x+\delta)) \d u\Big] N(x + \delta) + o(\|\delta\|).
\end{align}
Notice that among all the terms in \eqref{telescoping}, $I_2$, $I_4$ and $I_5$ are all of order $O(\|\delta\|)$, and
\begin{align}
f(x) - f(x + \delta) = - \delta^\top \nabla f(x) + o(\|\delta\|) = O(\|\delta\|),
\end{align}
while only $I_1$ and $I_3$ contain the factor $[\tau(x) - \tau(x+\delta)]$. It follows from \eqref{dot_gx_bound} and \eqref{tau_bound} that  
\begin{align}
& |I_1| > \frac{1}{2}c_0 |\tau(x) - \tau(x+\delta)|,\\
& |I_3| < \frac{4\eta}{c_0} \lambda |\tau(x) - \tau(x+\delta)|,
\end{align}
where $\lambda:=\sup_{x\in\cV}\|\nabla^2 f(x)\|$. When $\eta < \frac{1}{16}\lambda$,  we have
\begin{align}
O(\|\delta\|) = |[f(x) - f(x+\delta)] - I_2 - I_4 - I_5| = |I_1 + I_3| \geq |I_1| - |I_3| > \frac{1}{4} c_0 |\tau(x + \delta) - \tau(x)|,
\end{align}
and therefore $|\tau(x + \delta) - \tau(x)| \to 0$ as $\|\delta\|\to 0$, that is, $\tau$ is continuous on $\cG$. 

We proceed to show $\tau$ is continuously differentiable. 
Using the shorthand 
\begin{align}
& r_2(x) = \int_0^1 u \nabla^2 f(x - u\tau(x) N(x)) \d u N (x) , \\
& r_3(x) = \int_0^1 \nabla^2 f(x - u\tau(x) N(x)) \d u N (x),
\end{align}
and the fact $\tau(x+\delta) = \tau(x) + o(\|\delta\|)$, we can further write
\begin{align}
&I_2 = -\tau(x) \delta^\top r_1(x) + o(\|\delta\|),\\
&I_3 = - \tau(x) [\tau(x) - \tau(x+\delta)] N(x)^\top r_2(x) + o(\|\delta\|),\\
& I_4 = \tau^2 (x) \delta^\top \nabla N(x)^\top r_2(x) + o(\|\delta\|),\\
&I_5 = -\tau(x) \delta^\top  r_3(x) + o(\|\delta\|).
\end{align}
Putting all these together into \eqref{telescoping}, we obtain
\begin{align}
& [- \dot g(\tilde s(x)) - \tau(x) N(x)^\top r_2(x)][\tau(x) - \tau(x+\delta)] \nonumber\\
&= \delta^\top [-  \nabla f(x) + \tau(x)  r_1(x) + \tau(x) r_3(x)  - \tau^2(x) \nabla N(x)^\top r_2(x)] + o(\|\delta\|).
\end{align}
Again it follows from \eqref{dot_gx_bound} and \eqref{tau_bound} that $- \dot g(\tilde s(x)) - \tau(x) N(x)^\top r_2(x)>0$ when $\eta$ is small enough, and the above approximation becomes
\begin{align}
\tau(x+\delta) - \tau(x) = \delta^\top a(x)+ o(\|\delta\|),
\end{align}
where
\begin{align}
a(x) := \frac{\nabla f(x) - \tau(x) [r_1(x) + r_3(x)]  + \tau^2(x)\nabla N(x)^\top  r_2(x)}{- \dot g(\tilde s(x)) - \tau(x) N(x)^\top r_2(x) },
\end{align}
This implies that derivatives of $\tau$ exist, and $\nabla \tau(x) = a(x)$, which is continuous. 
Let $\cT_x\cL_{t+\eta}$ be the tangent space of $\cL_{t+\eta}$ at $x$. For any $v\in\cT_x\cL_{t+\eta}\setminus\{0\}$, 
\begin{align}
[\nabla Q(x) - \I]v = \tau(x) \Big[ \nabla N(x) - N(x) \frac{- [r_1(x) + r_3(x)]^\top  + \tau(x)\nabla r_2(x)^\top N(x)  }{- \dot g(\tilde s(x)) - \tau(x) N(x)^\top r_2(x) } \Big] v,
\end{align}
where we have used the fact $\nabla f(x)^\top v =0$. 
Hence by \eqref{tau_bound}, $\| [\nabla Q(x) - \I]v \| < \|v\|$ for $\eta$ small enough, which then implies $\nabla Q(x) v_1 \neq \nabla Q(x) v_2$ for any $v_1,v_2\in\cT_x\cL_{t+\eta}$ such that $v_1 \neq v_2,$ that is, $\nabla Q(x)$ is invertible as a linear map from $\cT_x\cL_{t+\eta}$ to $\cT_{Q(x)} \cL_{t}$. By the Inverse Function Theorem, $Q$ is a local diffeomorphism. Since $Q$ has been proved to be a homeomorphism, it is also a diffeomorphism. We then conclude that the metric projection from $\cL_t$ to $\cL_{t+\eta}$, as the inverse of $Q$, is also a diffeomorphism. 
\end{proof}

\begin{thm}
\label{thm:diffeomorphism2}
The transformation $x \mapsto \zeta_x(s-t)$ is a diffeomorphism from $\level_t$ to $\level_s$ whenever there are no critical points at any level anywhere between $t$ and $s$, inclusive.
\end{thm}

\begin{proof}
For some $\beta\in(0,s-t)$ small enough, consider an open subset of $\bbR^d$
\begin{equation}
\cH:= \bigcup_{u\in(t - \beta,t+\beta)} \level_u = \up_{t-\beta}^\circ \setminus \up_{t + \beta},
\end{equation}
such that $\inf_{x\in \bar{\cH}}\|\nabla f(x)\|>c_0$ for some $c_0>0$. We extend $\psi$ defined in the proof of \thmref{zeta} to $\cH$ in the following way: for $x\in \cH$, define
\begin{equation}
\psi (x) = \zeta_{x}(s-f(x)).
\end{equation}
Note that $\psi(x)\in \cL_s$ using the property of $\zeta$ stated in \thmref{zeta}. Consider $\delta\in\bbR^d$ with $\|\delta\|$ small enough such that $B(x,\|\delta\|) \subset \cH$, and $|f(x)-f(x+\delta)|<s-t$. Without loss of generality, suppose that $f(x+\delta)\le f(x)$. We can write
\begin{align}
\label{psi_delta_diff}
\psi(x + \delta) - \psi(x) &= [\zeta_{x+\delta}(s-f(x+\delta)) - \zeta_{x+\delta}(s-f(x))] + [\zeta_{x+\delta}(s-f(x)) - \zeta_{x}(s-f(x))] \nonumber\\
&=: I_1 + I_2.
\end{align}
For $I_1$, we have
\begin{align}
I_1 & = \Big[(x+\delta) + \int_{0}^{s-f(x+\delta)} F(\zeta_{x+\delta}(u)) du\Big] - \Big[(x+\delta) + \int_{0}^{s-f(x)} F(\zeta_{x+\delta}(u)) du\Big] \\
& = \int_{s-f(x)}^{s-f(x+\delta)} F(\zeta_{x+\delta}(u)) du.
\end{align}
Recall that $\cV= \up_t \setminus \up_s^\circ$. Note that for any $u\in[s-f(x), s-f(x+\delta)]$, we have $\zeta_{x+\delta}(u)\in\cV$, and hence
\begin{align}
\|F(\zeta_{x+\delta}(u)) - F(\zeta_{x+\delta}(s-f(x)))\| &= \Big\|\int_{s-f(x)}^{u}\nabla F(\zeta_{x+\delta}(w))\dot\zeta_{x+\delta}(w)\d w\Big\| \\
&\le \sup_{x\in \cV} \|\nabla F(x) F(x)\| |u- (s-f(x))| \\
&\le \sup_{x\in \cV} \|\nabla F(x) F(x)\| |f(x+\delta) - f(x)|\\
&\le \sup_{x\in \cV} \|\nabla F(x) F(x)\| \sup_{x\in\bbR^d}\|\nabla f(x)\| \|\delta\|.
\end{align}
Under the assumption that $f$ is twice continuously differentiable, using the same argument for \eqref{psi_diff_bound}, we can show that $\|\zeta_{x+\delta}(s-f(x)) - \zeta_{x}(s-f(x))\| = O(\|\delta\|)$, and hence $\|F(\zeta_{x+\delta}(s-f(x))) - F(\zeta_{x}(s-f(x)))\| = O(\|\delta\|)$. Therefore
\begin{align}
\label{I1_approximation}
I_1 & = [F(\zeta_{x}(s-f(x))) + O(\|\delta\|)] [f(x) - f(x+\delta)] \nonumber\\
&= [F(\zeta_{x}(s-f(x)))\nabla f(x)^\top]\delta  + o(\|\delta\|) \nonumber\\
&= : A(x) \delta  + o(\|\delta\|).
\end{align}
For $I_2$, we use \citep[Thm 18, Ch 4]{pontriagin1962ordinary} and get
\begin{align}
\label{I2_approximation}
I_2 = B(x)\delta + o(\|\delta\|),
\end{align}
where $B(x)=[b_1(x),\cdots,b_d(x)]$ is a $d\times d$ matrix. Here for $i=1,\cdots,d$, $b_i(x) = \xi_{i}(s-f(x))$, where $\xi_i$ satisfies
\begin{align}
\dot \xi_{i}(t) = \nabla F(\zeta_x(t)) \xi_{i}(t),\; t\in[0, s-f(x)]; \; \xi_{i}(0) = e_i,
\end{align}
where $e_i$ is a $d$-dimensional vector with the $i$th component 1 and all other components 0. Note that $B$ is continuous, because both $\xi_i$ and $f$ are continuous. 

Hence it follows from \eqref{psi_delta_diff}, \eqref{I1_approximation}, and \eqref{I2_approximation} that $\psi$ is differentiable with $\nabla \psi(x) = A(x) + B(x)$, which is continuous. For $x\in\cL_t$ and any $v\in\cT_x\cL_t\setminus\{0\}$, $\nabla \psi(x) v = B(x)v$ because $\nabla f(x)^\top v =0$. Note that $B(x)v = \xi(s-f(x))$, where $\xi$ satisfies
\begin{align}
\dot \xi(t) = \nabla F(\zeta_x(t)) \xi(t),\; t\in[0, s-f(x)]; \; \xi(0) = v.
\end{align}
This is a system of nonautonomous linear ODEs, for which \cite[Corollary 6, page 246]{cooke1992ordinary} guarantees that $\xi(s-f(x))\neq 0$. This then implies $\nabla \psi(x) v_1 \neq \nabla \psi(x) v_2$ for any $v_1,v_2\in\cT_x\cL_{t}$ such that $v_1 \neq v_2$. We conclude the proof by using the same argument as the last part of the proof of \thmref{diffeomorphism}.
\end{proof}


\subsection*{Acknowledgments}
WQ's work was partially supported by NSF Grant No. 1821154.

\bibliographystyle{chicago}
\bibliography{ref}

\end{document}